\documentclass[12pt,a4paper]{amsart}
\usepackage{amssymb,latexsym,amsxtra,amscd,amsfonts}
\usepackage{enumerate}
\usepackage[mathscr]{eucal}
\usepackage{dsfont}
\usepackage{calrsfs}
\usepackage{fancyhdr}
\usepackage{graphicx}
\usepackage{pst-all}
\usepackage{mathtools}
\definecolor{britishracinggreen}{rgb}{0.0, 0.26, 0.15}
\definecolor{amaranth}{rgb}{0.9, 0.17, 0.31}

\numberwithin{equation}{section}

\newtheorem{thm}{Theorem}
\newtheorem{cor}[thm]{Corollary}
\newtheorem{lem}[thm]{Lemma}
\newtheorem{prop}[thm]{Proposition}

\newtheorem{remark}{Remark}

\theoremstyle{definition}
\newtheorem{defn}{Definition}
\newtheorem{exmp}[thm]{Example}

\newcommand{\set}[2]{\ensuremath{\left\{ #1\,:\; #2\right\}}}

\newcommand{\sss}[1]{{\scriptscriptstyle #1}}

\newcommand{\converges}[1]{\mathop{\overset{\text{#1}}{\longrightarrow}}}

\newcommand{\N}{\mathds N}
\newcommand{\R}{\mathds R}

\newcommand{\lattice}{L}

\newcommand{\aaa}{\mathscr A}

\newcommand{\mm}{\mathscr M}
\newcommand{\nn}{\mathscr N}

\newcommand{\lbigvee}[1]{\bigvee^{\scriptscriptstyle #1}}
\newcommand{\lbigwedge}[1]{\bigwedge^{\scriptscriptstyle #1}}

\newcommand{\DM}{\text{DM}}

\begin{document}

	\title{Unbounded order convergence on infinitely distributive lattices}
	
	\author{Kevin Abela}
	\address{Kevin Abela \\ Department of Mathematics \\Faculty of Science \\University of Malta \\Msida MSD 2080, Malta} 
\email {kevin.abela.11@um.edu.mt}
	
	\author{Emmanuel Chetcuti}
	\address{Emmanuel Chetcuti\\Department of Mathematics\\Faculty of Science\\University of Malta\\Msida MSD 2080  Malta} 
\email {emanuel.chetcuti@um.edu.mt}

	\date{\today}
	


\begin{abstract}
We study $\mathfrak{u}$O-convergence on infinitely distributive lattices, extending key properties known from Riesz spaces. We show that order continuity of $\mathfrak{u}$O-convergence characterizes infinite distributivity. We examine O-adherence and $\mathfrak{u}$O-adherence of sublattices and ideals, proving that the $\mathfrak{u}$O- and O-closures of a sublattice coincide and form a sublattice, and that the first $\mathfrak{u}$O-adherence of an ideal is an O-closed ideal. We also analyze the Dedekind–MacNeille completion of a sublattice $Y$ within that of a lattice $L$, identifying conditions (A) and (B) under which the completion of $Y$ embeds regularly in that of $L$. In this case, we show that the first $\mathfrak{u}$O-adherence of $Y$ covers its O-closure.
\end{abstract}

\maketitle

\section{Introduction}

In the literature, order convergence has been thoroughly studied on Riesz spaces, lattices and partially ordered sets \cite{GAOLEU17,MathewsAnderson1967, Rennie1951, WOLK1961}. Thus, over the years, one can find several different definitions of O-convergence. The interested reader can look at \cite{KAECWH} to see different definitions and under which conditions these definitions agree or differ. A concept closely related to order convergence is unbounded order convergence. Unbounded order convergence ($\mathfrak{u}O$-convergence) was first introduced by Nakano under the name individual convergence \cite{HAK41, HAK48}. Later, DeMarr coined the commonly used term unbounded order convergence \cite{DEMARR}.  For sequences, $\mathfrak{u}O$-convergence is generally studied on Riesz spaces due to its natural relation to pointwise convergence; for sequences in $L_{p}(\mu)$ for $1 \leq p < \infty$ and finite measure $\mu$, $\mathfrak{u}O$-convergence is also equivalent to convergence almost everywhere. This relationship between almost everywhere convergence and $\mathfrak{u}O$-convergence was further investigated by Wickstead in \cite{WIC77}. He studied $\mathfrak{u}O$-convergence and weak convergence in Banach lattices, and showed that for norm bounded nets, weak and $\mathfrak{u}O$-convergence are equivalent. Kaplan studied $\mathfrak{u}O$-convergence on Riesz spaces with a weak order unit \cite{kaplan}. He showed that in a Riesz space with a weak order unit, $\mathfrak{u}O$-convergence has a simpler form. This form was used to give a new proof of a result by Hakano.

Gao and Xanthos showed that every weakly compact $\mathfrak{u}O$-convergent net is norm convergent in Banach lattices with the positive Schur property. The notion of $\mathfrak{u}O$-Cauchy nets was used to show that every relative weakly compact $\mathfrak{u}O$-Cauchy net is $\mathfrak{u}O$-convergent in an order continuous Banach lattice \cite{GNX14}. Gao studied $\mathfrak{u}O$-convergence in the dual of Banach spaces. He showed that every norm bounded $\mathfrak{u}O$-convergence net in $X^{*}$ is $w^{*}$-convergent if and only if $X$ has order continuous continuous norm. Furthermore, every $w^{*}$-convergent net in $X^{*}$ was shown to be $\mathfrak{u}O$-convergent if and only if $X$ is atomic with order continuous norm \cite{GAO14}.

A pivotal study on $\mathfrak{u}O$-convergence was done by Gao, Xanthos and Troitsky \cite{GTX17}. They proved that $\mathfrak{u}O$-convergence passes freely to and from regular Riesz sublattices. This was eventually used to improve several results in \cite{GAO14,GNX14}. They also proved that a Riesz sublattice $Y$ in a Riesz space $X$, is $O$-closed if and only if it is $\mathfrak{u}O$-closed. The relationships between $\mathfrak{u}O$-closure and $O$-closure was further investigated in \cite{GAOLEU17}. Bilokopytov and Troitsky studied $\mathfrak{u}O$-convergence in spaces of continuous functions, in particular $C(X), C_{b}(X), C_{0}(X)$ and $C^{\infty}(X)$ where $X$ is a completely regular Hausdorff topological space. They characterized $\mathfrak{u}O$-convergence in $C(X)$. Furthermore, they proved that a sequence $\mathfrak{u}O$-converges if and only if it converges pointwise on a co-meagre set \cite{TroitskyBilokopytov}.

In this article, we investigate $\mathfrak{u}$O-convergence on infinitely distributive lattices. We demonstrate that several properties that hold for Riesz spaces can also be studied in this broader setting. In section 3, we prove that $\mathfrak{u}$O-convergence being order continuous is equivalent to the lattice being infinitely distributive. In section 4, we study the O-adherence and $\mathfrak{u}$O-adherence of sublattices and ideals. In particular, we prove that the $\mathfrak{u}$O-closure and O-closure of a sublattice coincide and, the resulting set is itself a sublattice. Additionally, we prove that for ideals, the first $\mathfrak{u}$O-adherence is O-closed. Finally, we investigate the Dedekind MacNeille completion of a sublattice $Y$ with respect to the MacNeille completion of a lattice $L$. We identify two properties (A) and (B), under which the MacNeille completion of $Y$ embeds regularly in the MacNeille completion of $L$. Moreover, we prove that if a sublattice satisfies properties (A) and (B), then the first $\mathfrak{u}$O-adherence is O-closed.

\section{Notation}

A subset $D$ of a partially ordered set $P$ is \textit{directed} (resp. \emph{filtered}) provided it is non-empty and every finite subset of it has an upper bound (resp. lower bound) in $D$. For any two elements $s,t \in P$ such that $s \leq t$, we denote by $[s,t]$ the interval $\{x \in P : s \leq x \leq t\}$. For $A \subseteq P$ we denote by $\lbigvee{P} A$  the supremum of $A$ in ${P}$ (when this exists). Dually, we write $\lbigwedge{P} A$ for the infimum.  When it is clear in which space we are taking the supremum/infimum, we simply write $\bigvee A$ or $\bigwedge A$.  A \emph{lattice} is a partially ordered set in which every finite subset has an infimum and a supremum.  We shall normally denote a lattice by  ${L}$. A subset $Y \subseteq L$ is said to be a \textit{sublattice} if $a \vee^{\scriptscriptstyle L} b \in Y$ and $a \wedge^{\scriptscriptstyle L} b \in Y$ for every every $a,b \in Y$. A function $f$ from a lattice $K$ into a lattice $L$ is said to be a \textit{lattice-homomorphism} if $f(a\vee^{\sss K} b)=f(a)\vee^{\sss L} f(b)$ and $f(a\wedge^{\sss K} b)=f(a)\wedge^{\sss L}f(b)$ for every $a,b\in K$.  If the lattice-homomorphism  $f$ is injective, then we say that $f$ is an \emph{embedding}.  A sublattice $Y$ of $L$ is said to be \textit{regular}  if the inclusion embedding of $Y$ in $L$ preserves arbitrary infima and suprema.   A lattice $L$ satisfying $a \wedge (b \vee c) = (a \wedge b) \vee (a \wedge c)$ for every $a,b,c \in L$ is said to be a \textit{distributive lattice}.

Let $(x_\gamma)_{\gamma\in \Gamma}$ be a net in a set $X$. If $\Gamma'$ is some directed set and $\varphi:\Gamma'\to\Gamma$ is increasing and final\footnote{i.e. for every $\gamma\in \Gamma$ there exists $\gamma'\in\Gamma'$ such that $\varphi(\gamma')\ge \gamma$}, then the net $(x_{\varphi(\gamma')})_{\gamma'\in\Gamma'}$ is called a \emph{subnet} of $(x_\gamma)_{\gamma\in \Gamma}$.
For a general property of nets, say $P$,   we say that  \emph{$(x_\gamma)_{\gamma\in\Gamma}$ satisfies $P$ eventually}, if there exists a $\gamma_0\in \Gamma$ such the subnet
$(x_\gamma)_{\Gamma\ni\gamma\ge\gamma_0}$ has the property $P$.  If $(x_\gamma)_{\gamma\in\Gamma}$ is  increasing, its supremum exists and equals $x$, we write $x_\gamma\uparrow x$.  Dually, $x_\gamma\downarrow x$ means that the net $(x_\gamma)_{\gamma\in\Gamma}$ is  decreasing with infimum equal to $x$.

\section{O-convergence and $\mathfrak{u}$O-convergence}

In this section, we introduce the notions of O-convergence and $\mathfrak{u}$O-convergence within the framework of lattices. While O-convergence has been widely studied in the broader context of general lattices and partially ordered sets (posets) \cite{Erne1980,Gingras,MathewsAnderson1967, Rennie1951,Ward1955,WOLK1961}, the more recent concept of $\mathfrak{u}$O-convergence has predominantly been investigated in Riesz spaces and $\ell$-groups. Given that both Riesz spaces and $\ell$-groups are examples of infinitely distributive lattices, our work not only generalizes existing results on $\mathfrak{u}$O-convergence in these settings but also offers potential insights into the influence of the additive structure on the behavior of O- and $\mathfrak{u}$O-convergence.

\begin{prop}\cite[Proposition 3.10]{KAECWH2}\label{a}
	\begin{enumerate}[{\rm(a)}]
	\item For a lattice $L$ the following statements are equivalent.
	\begin{enumerate}[{\rm(i)}]
		\item  $L$ is distributive.
		\item For all $s,t\in L$ the function $f_{s,t}: L \to L$ defined by $f_{s,t}(x):= (x \wedge t) \vee s$ is a lattice homomorphism.
		\item  For all $s,t\in L$ the function $g_{s,t}: L \to L$ defined by  $g_{s,t}(x):= (x \vee s) \wedge t$ is a lattice homomorphism.
	\end{enumerate}
\item  If $L$ is distributive, then
$f_{s,t}=f_{s,s\vee t}=g_{s,s\vee t}$ and  $g_{s,t}=g_{s\wedge t,t}=f_{s\wedge t,t}$
for every $s,t\in L$.
\end{enumerate}
\end{prop}

\begin{defn}\label{d1}
Let $(x_\gamma)_{\gamma\in \Gamma}$ be a net and  $x$ a point in a lattice $L$.
\begin{enumerate}[{\rm(i)}]
\item  $(x_\gamma)_{\gamma \in \Gamma}$ is said to \emph{order converge} (\emph{O-converge}) to $x\in L$ if  there exists a directed set $M\subseteq L$ and a filtered  set $N\subseteq L$ satisfying $\bigvee M=\bigwedge N=x$, and such that for every $(m,n)\in M\times N$, $(x_{\gamma})_{\gamma\in\Gamma}$ is eventually contained in $[m,n]$.  In this case we write $x_\gamma\converges{O}x$.
    \item $(x_\gamma)_{\gamma \in \Gamma}$ is said to \emph{unbounded order converge} (\emph{$\mathfrak{u}$O-converge}) to $x \in L$, if $ (x_{\gamma} \wedge t) \vee s  \converges{O} (x \wedge t) \vee s$ for every $s,t \in L$ and $s \leq t$.  In this case we write $x_\gamma\converges{$\mathfrak{u}O$}x$.
\end{enumerate}
\end{defn}

The subsequent remark presents a collection of immediate consequences derived from these definitions.

\begin{remark}\label{remark}
The following assertions are easily verified. 
\begin{enumerate}[{\rm(i)}]
    \item If $x_{\gamma}\uparrow x$ in  $L$ then $x_{\gamma} \converges{O} x $.  The dual statement for decreasing nets holds as well.  \\    
    \item If a net O-converges, then the order limit is unique. Let us verify that if a net $\mathfrak{u}$O-converges, then the $\mathfrak{u}$O-limit is unique.  Indeed, if $x_\gamma\converges{$\mathfrak{u}O$}x$ and $x_\gamma\converges{$\mathfrak{u}O$}y$, then $(x_\gamma\wedge t)\vee s\converges{O} x$ and $(x_\gamma\wedge t)\vee s\converges{O} y$, for every $s\le t$ in $L$.  This implies that $(x\wedge t)\vee s=(y \wedge t)\vee s$ for every $s\le t$ in $L$.  In particular, setting $s:=x\wedge y$ and $t:=x\vee y$, one gets
        \[x=(x\wedge(x\vee y))\vee(x\wedge y)=(y\wedge(x\vee y))\vee(x\wedge y)=y\,.\]
     \smallskip
    \item If $(x_\gamma)_{\gamma\in \Gamma}$ is O-convergent to $x$, and eventually $(x_\gamma)_{\gamma\in\Gamma}$ is contained in $(\leftarrow, a]$, then $x\le a$.  The dual statement holds as well.\\
    \item If $L$ is bounded,  $\mathfrak u$O-convergence  implies O-convergence.\\
    \item In the light of Proposition \ref{a},  if the lattice is distributive, the condition $s \leq t$ in the definition of $\mathfrak{u}O$-convergence becomes redundant: $x_\gamma\converges{$\mathfrak u$O}x$ iff $(x_{\gamma} \vee s) \wedge t  \converges{O} (x \vee s) \wedge t$ for every $s,t \in L$.  
   \end{enumerate}
 \end{remark}

The following proposition demonstrates that when the lattice is a commutative $\ell$-group (and, in particular, when it is a Riesz space), the unbounded-order convergence defined in Definition \ref{d1} coincides with the well-established notion of unbounded convergence on such structures.

\begin{prop}\cite[Prop. 7.2]{Papangelou Fredos}\label{2}
	For the net $(x_\gamma)_{\gamma \in \Gamma}$, and the point $x$, in a commutative $\ell$-group $(\mathscr G,+)$, the following  statements are equivalent:
 \begin{enumerate}[{\rm(i)}]
     \item $\vert x_{\gamma} - x \vert\wedge u  \converges{O} 0 $ for every $u \in \mathscr G_{+}$,
     \item $(x_{\gamma} \wedge t) \vee s  \converges{O} (x \wedge t) \vee s$ for every $s,t \in \mathscr{G}$.
 \end{enumerate}
\end{prop}

\begin{prop}\label{38}
	Let $(x_\gamma)_{\gamma\in\Gamma}$ be a net in a  lattice $L$.  
	\begin{enumerate}[{\rm(i)}]
	\item If $x_{\gamma} \converges{$\mathfrak{u}$O} x$,  and eventually $(x_\gamma)_{\gamma\in\Gamma}$ is contained in $(\leftarrow, u]$, then $x\le u$.  The dual statement holds as well.
	\item If $(x_{\gamma})_{\gamma \in \Gamma}$ is monotonic, the following implication holds:	
	\[x_{\gamma} \converges{$\mathfrak{u}$O} x\quad\Longrightarrow\quad \begin{cases}
x_{\gamma}\uparrow x\text{ (if the net is increasing)},\\
x_{\gamma}\downarrow x\text{ (if the net is decreasing)}.
\end{cases}\]
	\end{enumerate}
\end{prop}
\begin{proof}
{\rm(i)}~Assume that $x_{\gamma} \converges{$\mathfrak{u}$O} x$, and $x_\gamma\le u$ for every $\gamma\ge \gamma_0$.  Define $a:=u\wedge x$ and $b:=u\vee x$.  We have $(x_\gamma\wedge b)\vee a\converges{O}(x\wedge b)\vee a$,  and by  Remark \ref{remark}{\rm(iii)} we obtain $x=(x\wedge b)\vee a\le(u\wedge b)\vee a=u$.  The dual statement can be proved similarly.  

{\rm(ii)}~ Assume that $(x_\gamma)_{\gamma\in\Gamma}$ is monotonic increasing and $x_{\gamma} \converges{$\mathfrak{u}$O} x$.  For $\gamma_{0} \in \Gamma$ the net $(x_\gamma)_{\gamma\ge \gamma_0}$ is contained in $[x_{\gamma_0}, \rightarrow)$ and $\mathfrak u$O-convergent to $x$.  {\rm(i)} implies  $x\ge x_{\gamma_0}$.    Therefore
$x$ is an upper bound for $(x_{\gamma})_{\gamma \in \Gamma}$.  Hence, $x_\gamma\converges{O}x$ and therefore $x=\bigvee_{\gamma\in\Gamma}x_\gamma$.

\end{proof}

The following example illustrates that the converse of Proposition \ref{38} (ii) may fail, even in the context of distributive lattices. This stands in sharp contrast to the case of Riesz spaces, where 
$\mathfrak{u}$O-convergence is order continuous\footnote{That is, if a net is order convergent to a point, then it is also 
$\mathfrak{u}$O-convergent to the same point.}. In such spaces, the desired implication follows directly from Remark \ref{remark} (i). As will become evident, the order continuity of 
$\mathfrak{u}$O-convergence is a property that arises specifically in infinitely distributive lattices.

\begin{exmp}
Let $L$ denote the collection of all the closed subsets of $\R$.  When endowed with set inclusion, $L$ forms a bounded distributive lattice. For $n\in\N$ let $X_n:=\left[{2^{-n}},\infty\right)$ and let $X:=[0,\infty)$. Then $(X_n)_{n\in\N}$ is increasing and $\lbigvee{L}X_n=X$, i.e. $X_n\uparrow X$ in $L$. In particular, $X_n\converges{\text{O}} X$.   On the  other hand, $(X_n)_{n\in\N}$ does not $\mathfrak{u}O$-converge to $X$.  To see this, let $A:=(-\infty,-1]$ and $B:=(-\infty,0]$.  Then  $(X_n\wedge B)\vee A=A$ for every $n\in\N$, i.e. $(X_n\wedge B)\vee A\converges{\text{O}} A$.  But $(X\wedge B)\vee A=\{0\}\cup A$.
\end{exmp}

\begin{defn}
\begin{enumerate}[{\rm(i)}]
    \item A lattice $L$ is said to satisfy the \textit{meet-infinite distributive law} if for $x\in L$ and $\{x_{\alpha} : \alpha \in \mathcal{A}\} \subseteq L$ such that $\bigwedge_{\alpha  \in \mathcal{A}} x_{\alpha}$ exists in $L$, 
	\begin{equation*}
		x \vee \bigwedge _{\alpha  \in \mathcal{A}} x_{\alpha} = \bigwedge_{\alpha \in \mathcal{A}} (x \vee x_{\alpha}).
	\end{equation*}
 
 \item A lattice $L$ is said to satisfy the \textit{join-infinite distributive law} if for $x \in L$ and $\{x_{\alpha} : \alpha \in \mathcal{A}\}\subseteq L$ such that $\bigvee_{\alpha  \in \mathcal{A}} x_{\alpha}$ exists in $L$,
	\begin{equation*}
		x \wedge \bigvee_{\alpha  \in \mathcal{A}} x_{\alpha} = \bigvee_{\alpha \in \mathcal{A}} (x \wedge x_{\alpha}).
	\end{equation*}

 \item  A lattice satisfying both the join and meet-infinite distributive laws is called an \textit{infinitely distributive lattice}.
\end{enumerate}
	\end{defn}

For subsets $A$, $B$ of a lattice $L$, we write $A\vee B$ to denote the set $\{a\vee b:a\in A,\,b\in B\}$.  ($A\wedge B$ is defined analogously.)

\begin{prop}\label{3}
	Let $L$ be an infinitely distributive lattice.  If $(x_{\gamma})_{\gamma \in \Gamma}  \converges{O} x$ and  $(y_{\omega})_{\omega \in \Omega} \converges{O} y$, then the net 
	\[\bigl((x_{\gamma} \vee y_{\omega}):(\gamma , \omega) \in \Gamma \times \Omega\bigr)\]
	O-converges to $x \vee y$, and dually,
	\[\bigl((x_{\gamma} \wedge y_{\omega}):(\gamma , \omega) \in \Gamma \times \Omega\bigr)\]
	O-converges to $x \wedge y$.
\end{prop}
\begin{proof}
	There are directed sets $M^{x}$, $M^{y}$, and filtered sets $N^{x}$, $N^{y}$, such that for $(a^{x},b^{x}) \in M^{x} \times N^{x}$, and for $(a^{y},b^{y}) \in M^{y} \times N^{y}$, one can find $\gamma(a^{x},b^{x})$, $\omega(a^{y},b^{y}) $ such that $x_{\gamma} \in [a^{x},b^{x}]$ for $\gamma \geq \gamma(a^{x},b^{x})$ and $y_{\omega} \in [a^{y},b^{y}]$ for $\omega \geq \omega(a^{y},b^{y})$. Define $M:= M^{x} \vee M^{y}$ and $N:= N^{x} \vee N^{y}$. Then, $M$ is directed, $N$ is filtered, $\bigvee M = x \vee y$, and by infinite distributivity,  $\bigwedge N= x\vee y$. Furthermore, for $a^{x} \vee a^{y} \in M^{x} \vee M^{y}$ and $b^{x} \vee b^{y} \in N^{x} \vee N^{y}$, it holds that $x_{\gamma} \vee y_{\omega} \in [a^{x} \vee a^{y}\,,\, b^{x} \vee b^{y}]$ for $(\gamma, \omega) \geq \gamma(a^{x},b^{x}) \times \omega(a^{y},b^{y})$.  The other assertion can be proved similarly.  
\end{proof}

\begin{cor}\label{3.1}
	Let $L$ be an infinitely distributive lattice.  If $(x_{\gamma})_{\gamma \in \Gamma}\converges{{$\mathfrak{u}$O}} x$ and  $(y_{\omega})_{\omega \in \Omega} \converges{{$\mathfrak{u}$O}} y$, then 
	\[\bigl((x_{\gamma} \vee y_{\omega}):(\gamma , \omega) \in \Gamma \times \Omega\bigr)\]
	$\mathfrak u$O-converges to $x \vee y$, and dually,
	\[\bigl((x_{\gamma} \wedge y_{\omega}):(\gamma , \omega) \in \Gamma \times \Omega\bigr)\]
	$\mathfrak u$O-converges to $x \wedge y$.
\end{cor}
\begin{proof}
	For any $s,t\in L$ we have
 \begin{align*}
     (x_\gamma\vee s)\wedge t\,&\converges{O}\,(x\vee s)\wedge t,\\ 
     (y_\omega\vee s)\wedge t\,&\converges{O}\,(y\vee s)\wedge t,
     \end{align*}
and therefore, by Proposition \ref{3}, 
\[((x_\gamma\vee s)\wedge t)\vee((y_\omega\vee s)\wedge t)\,\converges{O}\,((x\vee s)\wedge t)\vee((y\vee s)\wedge t),\]
and
\[((x_\gamma\vee s)\wedge t)\wedge((y_\omega\vee s)\wedge t)\,\converges{O}\,((x\vee s)\wedge t)\wedge((y\vee s)\wedge t).\]
Hence,
\[(x_\gamma\vee y_\omega\vee s)\wedge t\,\converges{O}\,(x\vee y\vee s)\wedge t,\]
and
\[((x_\gamma\wedge y_\omega)\vee s)\wedge t\,\converges{O}\,((x\wedge y)\vee s)\wedge t.\]
\end{proof}

We now show that for a distributive lattice $L$, the order continuity of $\mathfrak{u}$O-convergence is equivalent to $L$ being infinitely distributive.

\vskip 0.25 in

\begin{thm}\label{4}
	A distributive lattice $L$ is infinitely distributive if and only if  $\mathfrak{u}O$-convergence is order continuous.
\end{thm}
\begin{proof}
	If $L$ is infinitely distributive, then $\mathfrak{u}O$-convergence is order continuous by Proposition \ref{3}. Conversely, assume that $\mathfrak{u}O$-convergence is order continuous. Let $A \subseteq L$  and $y\in L$. We want to show that if $\bigvee A=u$ (resp. $\bigwedge A=v$) then $\bigvee \{a \wedge y : a \in A\} = u \wedge y$ (resp. $\bigwedge \{a\vee y:a\in A\} = v\vee y$). 
 Let $\aaa=\{B \subseteq A : \vert B \vert < \aleph_0\}$ and for every $B\in\aaa$ let $a_{B} = \bigvee B$, to get an increasing net $\{a_{B} : B \in \aaa\}$ in $L$ with supremum $u$. 
 Then $(a_B)_{B\in\aaa}$ is order convergent to $u$, and the assumed order continuity of $\mathfrak{u}O$-convergence implies that $a_{B} \converges{{$\mathfrak{u}\mbox{O}$}} u$. 
 
 Take $t = y$ and $s = a_0 \wedge y$, where $a_0$ is an arbitrary point of $A$.  Then
 \[(a_{B} \wedge y) \vee (a_0\wedge y) \converges{O} (u \wedge y)\vee(a_0\wedge y)=u\wedge y.\] 
 
 Observe that the net 
 \[\{(a_{B} \wedge y) \vee (a_0\wedge y):B\in\aaa\}\]
 is increasing, and so Proposition \ref{38} implies that 
 \[\bigvee\{(a_{B} \wedge y) \vee (a_0\wedge y):B\in\aaa\}=u\wedge y.\]
 Since $(a_{B} \wedge y) \ge (a_0\wedge y)$ when $a_0\in B$, it follows that 
 \[\bigvee\{(a_{B} \wedge y) \vee (a_0\wedge y):B\in\aaa\}=\bigvee\{(a_{B} \wedge y):B\in\aaa\}.\]
 Finally, if $B=\{a_1,\dots,a_n\}$, then $a_B\wedge y=\bigvee_{i\le n}(a_i\wedge y)$, because $L$ is distributive.  Thus 
 \[\bigvee \{a \wedge y : a \in A\}=\bigvee\{(a_{B} \wedge y):B\in\aaa\}=u\wedge y.\]
 
A similar argument can be made to prove the dual.
\end{proof}

\begin{prop}\label{37}
Let $L$ be an infinitely distributive lattice. If $(x_{\gamma})_{\gamma \in \Gamma}$ is an order bounded net, then $x_{\gamma} \converges{{$\mathfrak{u}$O}} x$ if and only if $x_{\gamma} \converges{O} x$.
\end{prop}
\begin{proof}
	That $x_{\gamma} \converges{O} x$ implies $x_{\gamma} \converges{{$\mathfrak{u}$O}} x$ follows by Theorem \ref{4}. On the  other hand assume that there exists $a,b \in L$ such that $x_{\gamma} \in [a,b]$ for every $\gamma \in \Gamma$ and $x_{\gamma} \converges{{$\mathfrak{u}$O}} x$. Then $(x_{\alpha} \wedge x_{\gamma}) \vee x_{\beta} \in [a,b]$ for every $\alpha$, $\beta$ and $\gamma$ in $\Gamma$, and thus $(x \wedge x_{\gamma}) \vee x_{\beta} \in [a,b]$. Repeating this argument it follows that $x \in [a,b]$. Hence, $x_{\gamma} = (x_{\gamma} \wedge b) \vee a \converges{O} (x \wedge b) \vee a = x$.
\end{proof}

\begin{cor}
  In a bounded, infinitely distributive lattice, O-convergence and $\mathfrak{u}$O-convergence are the same.
\end{cor}

\section{The O-closure and $\mathfrak{u}$O-closure of sublattices and ideals}

For a subset $ X$ of a lattice $L$ let
\[ X^{\sss{O}}_1:=\{x\in L:\mbox{there exists a net in $ X$ that O-converges to $x$}\}\,.\]
$ X^{\sss{O}}_1$ is called the \emph{1-O-adherence of $ X$}. 

We can define the  \textit{$\lambda$-O-adherence}  $ X^{\sss{{O}}}_\lambda$ recursively as follows. Set $ X^{\sss{O}}_0:= X$ and for every ordinal number $\lambda> 0$ define the  $\lambda$-O-adherence  $ X^{\sss{{O}}}_\lambda$ by:
\[ X^{\sss{{O}}}_\lambda:=\left(\bigcup_{\beta<\gamma} X^{\sss{{O}}}_\beta\right)^{\sss{O}}_1.\]

It is possible to define the \textit{$\lambda$-$\mathfrak u$O-adherence} $ X^{\sss{\mathfrak{u}{O}}}_\lambda$ for every ordinal $\lambda> 0$ by recursion, in the same way that we defined the $\lambda$-O-adherence  $ X^{\sss{{O}}}_\lambda$, i.e.\[ X^{\sss{\mathfrak{u}{O}}}_\lambda:=\left(\bigcup_{\beta<\lambda} X^{\sss{\mathfrak{u}{O}}}_\beta\right)^{\sss{\mathfrak{u}O}}_1.\]

The set $ X$ is said to be \textit{O-closed} (resp. \textit{$\mathfrak{u}$O-closed}) if $ X= X^{\sss{O}}_1$ (resp. $ X= X^{\sss{\mathfrak{u}O}}_1$).  The set of all O-closed subsets of $ L$ forms a topology on $ L$, called the \textit{order topology}. The same can be said for the $\mathfrak u$O-closed sets and one can speak of the\textit{ $\mathfrak u$O-topology} as the topology given rise by the $\mathfrak{u}$O-closed subsets of $ L$. The O-closure of $ X\subseteq L$ is the smallest O-closed subset of $ L$ that contains $ X$, i.e. the O-closure is the topological closure w.r.t. the order topology.  Note that this will generally be larger than $ X^{\sss{O}}_1$.  
Similarly, the $\mathfrak u$O-closure is the smallest $\mathfrak u$O-closed subset of $ L$ that contains $ X$.

\begin{remark}\label{rem2}
\begin{enumerate}[{\rm(i)}]
    \item  If $ X$ is a subset of an infinitely distributive lattice $ L$, then $ X^{\sss{O}}_1\subseteq X^{\sss{\mathfrak{u}O}}_1$, by Theorem~\ref{4}. \\
     \item In a lattice, the cuts $[a,\rightarrow)$ and $(\leftarrow,a]$ are O-closed (by Remark \ref{remark}~{\rm(iii)}), and  $\mathfrak{u}$O-closed (by Proposition \ref{38}~{\rm(i)}).    \\
  \item There must exist an ordinal $\lambda$ such that $ X^{\sss{O}}_\lambda= X^{\sss{O}}_{\lambda+1}$. To see this, it is enough to observe, for example, that $ X^{\sss{O}}_\kappa =  X^{\sss{O}}_{\kappa+1}$ when $\kappa$ is equal to any cardinality greater than the cardinality of $ L$.  The same holds for $\mathfrak{u}$O-adherence.  If we let 
    \begin{align*}
    \alpha:=&\min\{\lambda\ge 0: X^{\sss{O}}_\lambda= X^{\sss{O}}_{\lambda+1}\}\\
    \beta:=&\min\{\lambda\ge 0: X^{\sss{\mathfrak{u}O}}_\lambda= X^{\sss{\mathfrak{u}O}}_{\lambda+1}\}\,,
    \end{align*}
    then  $ X^{\sss{O}}_\alpha$ coincides with the  the O-closure (=topological closure w.r.t. the order topology) of $ X$ and $ X^{\sss{\mathfrak{u}O}}_\beta$ with the $\mathfrak u$O-closure (= topological closure w.r.t. the $\mathfrak{u}$O-topology) of $ X$.  When $ L$ is infinitely distributive, we note that $ X^{\sss{O}}_\alpha\subseteq  X^{\sss{\mathfrak{u}O}}_\beta$ (by Theorem \ref{4}).
\end{enumerate}
   \end{remark}

\begin{prop}\label{10.1}
Let $Y$ be a sublattice of an infinitely distributive lattice $ L$.  For every ordinal $\lambda\ge 0$ the sets $Y^{\sss{O}}_\lambda$ and $Y^{\sss{\mathfrak u}O}_\lambda$ are sublattices of $ L$.
\end{prop}
\begin{proof}
To prove this result, we use transfinite induction.  The assertion is trivially satisfied when $\lambda=0$.  For every $\lambda\ge 1$ define
\[A_\lambda:=\bigcup_{\beta<\lambda} Y^{\sss{O}}_\beta\qquad\text{and}\qquad B_\lambda:=\bigcup_{ \beta<\lambda}Y^{\sss{{\mathfrak u}O}}_\beta.\]

We shall prove only the statement regarding $O$-adherence; the proof for $\mathfrak{u}O$-adherence follows the same argument (applying Corollary \ref{3.1} instead of Proposition \ref{3}). 

For $x,\ y \in Y_{\lambda}^{\sss{O}}$  there exist two nets $(x_{\sigma})_{\sigma \in \Sigma}$ and $(y_\omega)_{\omega \in \Omega}$ in $A_\lambda$ such that $x_{\sigma} \converges{O} x$ and $y_{\omega} \converges{O} y$.  By the induction hypothesis, every $Y^{\sss{O}}_\beta$ $(0\le \beta<\lambda)$ is a sublattice of $ L$, and therefore so is $A_\lambda$.  Hence, $(x_{\sigma} \vee y_{\omega})_{(\sigma ,\omega) \in \Sigma \times \Omega}$ is a net in $A_\lambda$. Proposition \ref{3} implies $x_{\sigma} \vee y_{\omega} \converges{O} x \vee y$ and so  $x \vee y \in (A_\lambda)^{\sss{O}}_1=Y_{\lambda}^{\sss{O}}$.  Similarly, it can be shown that $x \wedge y \in Y_{\lambda}^{\sss{O}}$.
    \end{proof}

\begin{prop}\label{propInclusion}
Let $Y$ be a sublattice of an infinitely distributive lattice $ L$.  For every ordinal $\lambda\ge 0$ there exists an ordinal $f(\lambda)\ge \lambda$ such that 
\begin{equation}\label{e1}
    Y^{\sss{O}}_\lambda\subseteq Y^{\sss{\mathfrak u}O}_\lambda\subseteq Y^{\sss{O}}_{f(\lambda)}.
\end{equation}      
\end{prop}

\begin{proof}
We shall use the same notation introduced in the proof of Proposition \ref{10.1}.  The function $f$ shall be defined recursively.
Let $f(0):=1$.  Clearly (\ref{e1}) is satisfied.  Suppose that $\{f(\beta):\beta<\lambda\}$ have been constructed such that $Y^{\sss{O}}_\beta\subseteq Y^{\sss{{\mathfrak{u}O}}}_\beta\subseteq Y^{\sss{O}}_{f(\beta)}$ for every $\beta<\lambda$. The inclusion $A_{\lambda}\subseteq B_{\lambda}$ follows by the induction hypothesis and therefore
\[Y^{\sss{O}}_\lambda=(A_\lambda)^{\sss{O}}_1\subseteq (B_\lambda)^{\sss{O}}_1\subseteq (B_\lambda)^{\sss{\mathfrak{u}O}}_1=Y^{\sss{\mathfrak{u}O}}_\lambda,\]
where the last inclusion follows by Remark \ref{rem2} {\rm(i)}. Set $\omega:=\sup\{f(\beta):\beta<\lambda\}$. Then
\[B_\lambda=\bigcup_{\beta<\lambda}Y_\beta^{\sss{\mathfrak{u}O}}\subseteq\bigcup_{\beta<\lambda}Y_{f(\beta)}^{\sss{O}}\subseteq\bigcup_{\beta<\omega}Y^{\sss{O}}_\beta=A_\omega.\]
 Let $x\in (B_{\lambda})^{\sss{\mathfrak u}O}_1$.  There exists a net $(x_{\sigma})_{\sigma \in \Sigma}$ in $B_\lambda\subset A_\omega$ such that $x_{\sigma} \converges{{$\mathfrak{u}O$}} x$, i.e  $(x_{\sigma} \wedge t) \vee s \converges{O} (x \wedge t) \vee s$ for every $s,t \in  L$. For any triple $\sigma$, $\sigma'$ and $\sigma''$ in $\Sigma$, the element $(x_\sigma\wedge x_{\sigma'})\vee x_{\sigma''}$ lies in $A_\omega$ because the latter is a sublattice by Proposition \ref{10.1}. This implies that $(x\wedge x_{\sigma'})\vee x_{\sigma''}$ belongs to $(A_\omega)^{\sss{O}}_1=Y^{\sss{O}}_{\omega}$, for every $\sigma'$, $\sigma''$ in $\Sigma$.  By the same reasoning, it follows that $x\vee x_{\sigma''}\in Y^{\sss{O}}_{\omega+1}$ for every $\sigma''\in\Sigma$.  Using that $Y^{\sss{O}}_{\omega+1}$ is a sublattice, we note that $(x\vee x_{\sigma})\wedge(x\vee x_{\sigma'})\in Y^{\sss{O}}_{\omega+1}$ for every $\sigma,\sigma'\in\Sigma$.  Invoking once more the $\mathfrak{u}$O-convergence of $(x_\sigma)_{\sigma\in\Sigma}$ to $x$, it can be concluded that 
 $x=x\wedge(x\vee x_{\sigma'})$ belongs to $Y^{\sss{O}}_{\omega+2}$.  Define $f(\lambda):=\omega+2$.
\end{proof}

\begin{thm}\label{Thmsublattice}
The O-closure and the $\mathfrak u$O-closure of a sublattice of an infinitely distributive lattice $ L$ coincide, and the resulting subset is again a sublattice of ${L}$.
\end{thm}
\begin{proof}
Let $Y$ be a sublattice of $ L$.  Let 
\[\alpha:=\min\{\lambda\ge 0:Y^{\sss{O}}_\lambda=Y^{\sss{O}}_{\lambda+1}\}\quad\text{and}\quad\beta:=\min\{\lambda\ge 0:Y^{\sss{\mathfrak{u}O}}_\lambda=Y^{\sss{\mathfrak{u}O}}_{\lambda+1}\},\]
i.e. O-closure of $Y$ equals $Y^{\sss{O}}_\alpha$ and the $\mathfrak{u}$O-closure of $Y$ equals $Y^{\sss{\mathfrak{u}O}}_\beta$. These are lattices by Proposition \ref{10.1}, and $Y^{\sss{O}}_\alpha\subseteq Y^{\sss{\mathfrak{u}O}}_\beta$ by Theorem \ref{4}.  By definition of $\alpha$, observe that
\[Y^{\sss{O}}_\alpha=Y^{\sss{O}}_{f(\alpha)}=Y^{\sss{O}}_{f(f(\alpha))}\]
and therefore $Y^{\sss{O}}_\alpha=Y^{\sss{\mathfrak{u}O}}_\alpha=
Y^{\sss{\mathfrak{u}O}}_{f(\alpha)}$, by Proposition \ref{propInclusion}.  In particular, we get that $Y^{\sss{O}}_\alpha=Y^{\sss{\mathfrak{u}O}}_{\alpha}=Y^{\sss{\mathfrak{u}O}}_{\alpha+1}$, i.e. $Y^{\sss{O}}_\alpha$ is $\mathfrak{u}O$-closed.  Therefore $Y^{\sss{\mathfrak{u}O}}_\beta$ (the $\mathfrak{u}$O-closure of $Y$) is contained in $Y^{\sss{O}}_\alpha$.
\end{proof}

\begin{cor}\label{10}
	A sublattice of an infinitely distributive lattice is $O$-closed if and only if it is $\mathfrak{u}$O-closed.
\end{cor}

This example illustrates that, in Theorem \ref{Thmsublattice}, the condition of infinite distributivity is essential and cannot be replaced by the weaker assumption of distributivity.

\begin{exmp}
Consider the following subsets of $2^\R$.
\begin{align*}
\mathscr C_-:=&\{(-\infty,a]:a\le 0\}\\
\mathscr C'_-:=&\{(-\infty,a]:a< 0\}\\
\mathscr C_+:=&\{[a,+\infty):a\ge 0\}\\
\mathscr C'_+:=&\{[a,+\infty):a> 0\}
\end{align*}
The ring $ L$ of subsets of $\R$ generated by $\mathscr C_-\cup\mathscr C_+$
consists of all subsets of $\R$ that have one of the following types: $\emptyset$, $(-\infty,-a]$, $[b,+\infty)$, $(-\infty,-a]\cup [b,+\infty)$, $\{0\}$, where $a,b\ge 0$.
This forms a distributive lattice.  The sub-ring $Y$ generated by $\mathscr C'_-\cup\mathscr C'_+$ consists of all subsets of $\R$ that have one of the following types: $\emptyset$, $(-\infty,-a]$, $[b,+\infty)$, $(-\infty,-a]\cup [b,+\infty)$, where $a,b> 0$. $Y$ is a sublattice of $L$.   The O-closure $\overline Y$ of $Y$  in $L$ consists of the subsets of $\R$ that have one of following types: $\emptyset$, $(-\infty,-a]$,  $[b,+\infty)$,  $(-\infty,-a]\cup [b,+\infty)$, where $a,b\ge 0$.
Observe that the infimum in $\overline Y$ of $(-\infty,0]$ and $[0,+\infty)$ is equal to $\emptyset$, whereas the infimum taken in $L$ equals $\{0\}$.
\end{exmp}

Recall that a subset $A$ of a lattice $L$ is
 called a \emph{down-set} if, for every $a\in A$ and $x\in L$, the condition $x\le a$ implies $x\in A$.  Moreover, $A$ is an 
 \emph{ideal} if 
it is a down-set and closed under finite joins; that is,
 $a\vee b\in A$ for all $a, b\in A$.

\begin{prop}\label{11}
	Let $ A$ be a down-set in a lattice $ L$. Then, every $x \in  A_{1}^{O}$ is the supremum of an increasing net in $ A$. Moreover, if $ L$ is infinitely distributive, $ A_{1}^{O}$ is a down-set.
\end{prop}
\begin{proof}
Let $x \in  A_{1}^{O}$, then there exists a net $(x_{\gamma})_{\gamma \in \Gamma}$ in $A$ such that $x_{\gamma} \converges{O} x$. Then there exist a directed set $M$ and a filtered set $N$ such that $\bigvee \mm = x = \bigwedge \nn$ and for every $(m,n) \in M\times N$ the net $(x_{\gamma})_{\gamma \in \Gamma}$ is eventually in $[m,n]$. As $ A$ is a down-set, it follows that $M\subseteq  A$. Result follows from the fact that the set $M$ is directed, so it can be viewed as an increasing net indexed over itself.
  	
  	Furthermore, if $L$ is infinitely distributive, $a \in  A_{1}^{O}$ and $x \in  L$ with $x \leq a$. By the argument above, there exists an increasing net $(a_{\gamma})_{\gamma \in \Gamma} \subseteq  A$ such that $\bigvee_{\gamma \in \Gamma} a_{\gamma} = a$. Using the fact that $ A$ is a down-set, $\{a_{\gamma} \wedge x : \gamma \in \Gamma \} \subseteq  A$ and $\bigvee \{a_{\gamma} \wedge x : \gamma \in \Gamma \} = a \wedge x = x$ concluding that $x \in  A_{1}^{O}$.\end{proof}

\begin{thm}\label{22}
	Let $ L$ be an infinitely distributive lattice and $ A \subseteq  L$ be an ideal. Then $ A_{1}^{O}= A_{1}^{\mathfrak{u}O}$ and both are  $\mathfrak{u}$O-closed (and therefore O-closed) ideals.
\end{thm}
\begin{proof}
	
If $ A$ is an ideal, then $ A_{1}^{O}$ is an ideal by Proposition \ref{10.1} and Proposition \ref{11}.
 
Next we show that $ A_{1}^{O}$ is $O$-closed. By Proposition \ref{11}  and the fact that $ A_{1}^{O}$ is an ideal, for every $x \in  A_{2}^{O}$ there exists an increasing net $(x_{\gamma})_{\gamma \in \Gamma}$ in $ A_{1}^{O}$ such that $\bigvee_{\gamma \in \Gamma} x_{\gamma} = x$. Let $B:=\{a \in  A : \exists \gamma\in\Gamma\ \mbox{such that }a \leq x_{\gamma} \}$.  Then $B$ is non-empty, directed, and $x$ is an upper-bound of $B$. Finally, observe that if $k\in L$ and $b \leq k$ for every $b \in B$, then  $x_{\gamma} \leq k$ for every $\gamma \in \Gamma$, and therefore $x \leq k$.  This shows that  that $x$ is the supremum of $B$ and so  $x\in  A_{1}^{O}$. So $ A_2^O= A_1^O$ and therefore $ A_{1}^{O}$ is $O$-closed. This implies that  $ A_1^O$ is equal to the O-closure of $ A$.  Theorem \ref{Thmsublattice} yields that $ A_1^O$ is the $\mathfrak{u}$O-closure of $ A$, and therefore $ A_1^O= A_{1}^{\mathfrak{u}O}$.   
	\end{proof}

We now aim to establish an analogue of the previous theorem, replacing ideals with regular sublattices. Note that every ideal of a lattice is, in particular, a regular sublattice.

Given an infinitely distributive lattice $L$ and a regular sublattice $Y$ we ask: does the first $\mathfrak{u}O$-adherence $Y_{1}^{\mathfrak{u}O}$ cover the O-closure of $Y$?
In \cite[Theorem 2]{GAOLEU17}, a positive answer to this question is given in the special case where 
$L$ is an Archimedean Riesz space with the countable sup property and admitting a seperating family of order-continuous positive linear functionals, and $Y$ is a   Riesz subspace.
In this paper, we provide a positive answer under a purely order-theoretical condition, which we introduce below.

\begin{defn}
		Let $Y$ be a sublattice of a lattice $L$. Then $Y$ is said to have Property (A) if for $A \subseteq Y$ and $x \in A^{-}$, there exists $y \in A^{-}\cap {Y}$ such that $x \leq y$. Dually, $Y$ is said to have Property (B) if for $A \subseteq Y$ and $x \in A^{+}$, there exists $y \in A^{+}\cap{Y}$ satisfying $x \geq y$.
	\end{defn}

In Theorem \ref{n}, we  shall prove that if $L$ is an infinitely distributive lattice, and $Y \subseteq L$ a sublattice satisfying Properties (A) and (B), then $Y^{\sss{\mathfrak{u}{O}}}_1=Y^{\sss{{O}}}_1$, and both are simultaneously O-closed and $\mathfrak{u}$O-closed.

The remainder of the paper is dedicated to proving Theorem \ref{n}, which we accomplish by first establishing two intermediate results (Theorem \ref{28} and Theorem \ref{30})  that we believe may be of independent interest.

Let $P$ be a poset.  The set of upper-bounds  of the subset $ D$ of $P$  is denoted by $ D^{+}$ and the set of lower-bounds  is denoted by  $ D^{-}$.  If $D=D^{+-}$, then we say that $D$ is a lower-cut (l-cut) of $P$.  Clearly, $\emptyset$ and $P$ are l-cuts of $P$.    
Recall that the Dedekind-MacNeille completion of $ P$, denoted by $ \DM (P)$, is the set of all l-cuts of $P$, ordered by set inclusion.    $\DM (P)$ forms a \emph{complete lattice}  satisfying the following properties.
 \begin{enumerate}[{\rm(i)}]
 \item $ D^-$  belongs to $ \DM(P)$ for every $ D\subseteq P$.
 \item If $\set{ D_i}{i\in I}\subseteq \DM (P)$  then 
 \[\bigvee_{i\in I}{\!}^{\DM (P)}D_i=\left(\bigcup_{i\in I} D_i\right)^{+-}\quad\text{and}\quad\bigwedge_{i\in I}{\!}^{\DM (P)}D_{i}=\bigcap_{i\in I} D_{i}\,.\]
   \item  $(\leftarrow,x]\in \DM (P)$ for every $x\in P$ and the function $\varphi:P\to \DM(P)$ defined by $\varphi(x):=(\leftarrow,x]$ is an \emph{order-isomorphism}.
     \item $\varphi[ P]$ is \emph{join-dense} and \emph{meet-dense} in $\DM(P)$, i.e.
   \[ a =\bigvee{}^{\!{ \DM(P)}}\,\set{\varphi(x)}{x\in P,\,\varphi(x)\le  a}\,,\]
   and 
   \[ a =\, \bigwedge{}^{\! \DM(P)}\set{\varphi(x)}{x\in P,\,\varphi(x)\ge  a}\,,\]
   for every $ a\in \DM(P)$.  From this follows that  $\varphi$ preserves all suprema and infima that exist in $ P$, i.e. if $ D\subseteq P$ and $x\in P$, then 
   \begin{align*}
   \bigvee{}^{\! \DM(P)}\varphi[D]=\varphi(x)\ &\Leftrightarrow\ \bigvee{}^{\! P} D=x\\
    \bigwedge{}^{\! \DM(P)} \varphi[D]=\varphi(x)\ &\Leftrightarrow\ \bigwedge{}^{\! P}D=x.
    \end{align*}
   \item Let $D\subseteq P$.  Then
   \[ D^-=\bigwedge{}^{\! \DM(P)}\varphi[ D]=\bigvee{}^{\!{ \DM(P)}}\varphi[ D^-]\,,\]
    and 
    \[ D^{+-}=\bigvee{}^{\!{\DM(P)}}\varphi[ D]=\bigwedge{}^{\! \DM(P)}\varphi[ D^+]\,.\]
  \end{enumerate}
 
The Dedekind-MacNeille completion of $ P$ is characterized -- up to order-isomorphism -- as the unique complete lattice containing $ P$ as a simultaneously join-dense and meet-dense subset.  It is well-known that $\DM(P)$ need not satisfy the same lattice equations that are satisfied by $P$.  In In \cite{Crawley}, the author gives an  example of a distributive lattice which cannot be regularly imbedded in any complete modular lattice.  
Example \ref{exmp2} shows that not every infinitely distributive lattice has an infinitely distributive MacNeille completion.  This stands in sharp contrast to the situation in the setting of Riesz spaces and $\ell$-groups. If the poset $ P$ happens to be a commutative Archimedean $\ell$-group, it is possible to endow $P^\delta:=\DM(P)\setminus \{\emptyset, P\}$ with a group structure to obtain a Dedekind complete $\ell$-group, containing the starting $\ell$-group as a regular $\ell$-subgroup (see \cite{Clifford}).   Because $\ell$-groups are intrinsically infinitely distributive, it follows that their Dedekind–MacNeille completion is likewise infinitely distributive\footnote{Note that for a lattice $L$ without top and bottom elements, $L^\delta:=\DM(L)\setminus\{\emptyset,L\}$ is infinitely distributive if and only if $\DM(L)$ is infinitely distributive.}.

\begin{exmp}\label{exmp2}
 When endowed with the pointwise partial order, 
 \[{L}:= \{(0,b) : 0 \leq b < 1\} \cup  \{(1,b) : 0 \leq b<+\infty,\,b\neq 1\}\] forms an infinitely distributive lattice.  It is easy to see that
  \begin{align*}
  \DM(L)&= \{(0,b) : 0 \leq b < 1\} \cup  \{(1,b) : 0 \leq b \leq +\infty\}\,, \\
   L^\delta&=\{(0,b) : 0 < b < 1\} \cup  \{(1,b) : 0 \leq b < +\infty\}\,.
   \end{align*}
 Let us show that $L^\delta$ (and hence $\DM(L)$) does not satisfy the Join-Infinite Distributive Law.  Let $x_{n} = (0,1-\frac{1}{n})$. Then $\lbigvee{{L}^{\delta}} x_{n} = (1,1)$ and $\left(\lbigvee{{L}^{\delta}} x_{n}\right) \wedge (1,\frac{1}{2}) = (1,\frac{1}{2})$. On the  other hand, $\lbigvee{{L}^{\delta}} (x_{n}\wedge (1,\frac{1}{2})) = (0,\frac{1}{2})$. 
\end{exmp}

Let $Y$ be a sublattice of a lattice $\lattice$.  In Theorem \ref{28} we show that the MacNeille completion $\DM(Y)$ can be identified with a subset of $\DM(L)$.  

If $Y$ is a sublattice of a lattice $\lattice$ and $ D$ is a  subset of $Y$, define $ D^{+_{Y}}:= D^+\cap Y$ and $ D^{-_{Y}}:= D^-\cap Y$.


\begin{lem}\label{24}
Let $Y$ be a sublattice of a lattice $\lattice$.  If $A,B \subseteq Y$ such that $A^{+-} \subseteq B^{+-}$, then $A^{+_{Y} -_{Y} }\subseteq  B^{+_{Y} -_{Y}}$.
\end{lem}
\begin{proof}
	We start by showing that $B^{+_{Y}} \subseteq A^{+}$. Indeed, $B^{+_{Y}} \subseteq B^{+} = B^{+-+} \subseteq A^{+-+} = A^{+}$.
	
   Next note that $A^{+-_{Y}} \subseteq B^{+_{Y}-_{Y}}$. Indeed, let $y \in A^{+-_{Y}}$. Then $y \in Y$ and $y \leq b$ for every $b \in A^{+}$. From $B^{+_{Y}} \subseteq A^{+}$ it is clear that $y \leq c$ for every $c \in B^{+_{Y}}$. This concludes that $A^{+-_{Y}} \subseteq B^{+_{Y} -_{Y}}$. Finally from $A \subseteq A^{+-_{Y}} \subseteq B^{+_{Y} -_{Y}}$ it follows that $ A^{+_{Y} -_{Y}} \subseteq  B^{+_{Y} -_{Y}}$.
\end{proof}

\begin{lem}\label{25}
	Let $Y$ be a sublattice of a lattice $\lattice$.  For $A \subseteq Y$, $(A^{+-} \cap Y)^{+-} = A^{+-}$.
\end{lem}	
\begin{proof}
	First note that $(A^{+-} \cap Y) \subseteq A^{+-}$ implies that $(A^{+-} \cap Y)^{+-} \subseteq A^{+-}$. The other inclusion follows by noting that $A \subseteq A^{+-}$ and $A \subseteq Y$, then $A \subseteq A^{+-} \cap Y$ concludes that $A^{+-} \subseteq (A^{+-} \cap Y)^{+-}$.
\end{proof}

\begin{thm}\label{28}
	Let $L$ be a lattice and $Y \subseteq L$ be a sublattice. Then 
	\[i:\DM(Y)\to\DM(L): A\mapsto A^{+-} \]
	is an order-embedding of $\DM(Y)$ into $\DM(X)$.  
\end{thm}	
\begin{proof}
 This follows by Lemma \ref{24}.
\end{proof}

\begin{lem}\label{lem21}
	Let $L$ be a lattice and $Y$ be a sublattice satisfying (A) and (B). Then $Y$ is regular.
\end{lem}
\begin{proof}
	Let $A \subseteq Y$ and $b = \lbigvee{Y} A$. Assume that there exists $x \in A^{+}$, then by property (A) there exists $a \in A^{+_Y}$ such that $a \leq x$, thus concluding that $b = \lbigvee{L} A$. The dual statement can be proven similarly. 
\end{proof}	

In the next example, we note that the converse of Lemma \ref{lem21} is not true. 

\begin{exmp} \label{exmp3}
    Take $L = \{(0,b) : 0 \leq b \leq 1\} \cup  \{(1,b) : 0 \leq b \leq 1\}$ ordered point-wise and $L_{0} = \{(0,b) : 0 \leq b < 1\} \cup  \{(1,b) : 0 \leq b < 1\}$. Clearly $L$ is complete and $L_{0}$ is regular in $L$. However property (B) fails to hold.
\end{exmp}

\begin{prop}\label{26}
	Let $L$ be a lattice, $Y \subseteq L$ and $A \subseteq Y$.
\begin{enumerate}[(i)]
	\item Property $(A)$ implies that $A^{-+} = A^{-_{Y}+}$;
	\item Property $(B)$ implies that $A^{+-} = A^{+_{Y}-}$.
\end{enumerate}
\end{prop}
\begin{proof}
	We shall only show the proof of (i), as the proof for (ii) follows by a similar argument. It can be seen that $A^{-+} \subseteq A^{-_{Y}+}$. To see the reverse inclusion let $x \in A^{-_{Y}+}$ and $z \in A^{-}$. Then there exists $y \in A^{-_{Y}}$ such that $z \leq y$. This implies that $z \leq y \leq x$.  Hence, $x \in A^{-+}$.  This shows that $A^{-+} \supseteq A^{-_{Y}+}$.
	\end{proof}

\begin{prop}\label{27}
	Let $L$ be a lattice, $Y \subseteq L$ a sublattice of $L$ and $\{A_{\alpha} : \alpha \in \mathcal{A} \}$ a collection of sets in $Y$ satisfying $A_{\alpha}^{+_{Y}-_{Y}} = A_{\alpha}$.
	\begin{enumerate}[(i)]
		\item If $Y$ satisfies Property (A) then 
\[\bigcap_{\alpha \in \mathcal{A}} A_{\alpha}^{+-} = \left(\bigcap_{\alpha \in \mathcal{A}} A_{\alpha}\right)^{+-}\,.\]
		\item If $Y$ satisfies Property (B)  then 
\[\left(\bigcup_{\alpha \in \mathcal{A}} A_{\alpha}\right)^{+_{Y}-_{Y}+-} = \left(\bigcup_{\alpha \in \mathcal{A}} A_{\alpha}^{+-}\right)^{+-}\,.\]
	\end{enumerate}
\end{prop}
\begin{proof}
	\begin{enumerate}[(i)] 
		\item We first note that the inclusion $(\bigcap_{\alpha \in \mathcal{A}} A_{\alpha})^{+-} \subseteq \bigcap_{\alpha} A_{\alpha}^{+-}$ is trivial. To see the reverse inclusion, first note that $\bigcap_{\alpha \in \mathcal{A}} A_{\alpha}^{+-}  = (\bigcup_{\alpha \in \mathcal{A}} A_{\alpha}^{+})^{-}$. Furthermore,
		 \begin{align*}
			 \left(\bigcup_{\alpha \in \mathcal{A}} A_{\alpha}^{+}\right)^{-} &= \left(\bigcup_{\alpha \in \mathcal{A}} A_{\alpha}^{+}\right)^{-+-}\\
			&\subseteq \left(\bigcup_{\alpha \in \mathcal{A}} A_{\alpha}^{+_{Y}}\right)^{-+-} \\
			&=\left (\bigcup_{\alpha \in \mathcal{A}} A_{\alpha}^{+_{Y}}\right)^{-_{Y}+-} \text{ (By property (A))}\\
			&=\left (\bigcap_{\alpha \in \mathcal{A}} A_{\alpha}^{+_{Y}-_{Y}}\right)^{+-} \\
			&= \left(\bigcap_{\alpha \in \mathcal{A}} A_{\alpha}\right )^{+-}.
		\end{align*}
	\item To see this equality, we note that
	\begin{align*}
		\left(\bigcup_{\alpha \in \mathcal{A}} A_{\alpha}\right)^{+_{Y}-_{Y}+-} &= \left(\bigcup_{\alpha \in \mathcal{A}} A_{\alpha}\right)^{+_{Y}-_{Y} +_{Y}-} \text{ (By property (B))}\\
		&=\left (\bigcup_{\alpha \in \mathcal{A}} A_{\alpha}\right)^{+_{Y}-}\\
		&=\left (\bigcup_{\alpha \in \mathcal{A}} A_{\alpha}\right)^{+-} \text{ (By property (B))}\\
		&= \left(\bigcup_{\alpha \in \mathcal{A}} A_{\alpha}^{+-}\right)^{+-}.
	\end{align*}
	\end{enumerate}
\end{proof}

\begin{thm}\label{30}
	Let $L$ be a lattice and $Y \subseteq L$ be a sublattice.  Let $i:\DM(Y)\to \DM(L):A\mapsto A^{+-}$ be order-embedding described in Theorem \ref{28}.
	\begin{enumerate}[(i)]
		\item  If $Y$ satisfies Property (A), then $i$ preserves arbitrary meets.
\item If $Y$ satisfies Property (B), then $i$ preserves arbitrary joins.
	\end{enumerate}
\end{thm}
\begin{proof}
Let $\{A_{\alpha}: \alpha \in \mathcal{A} \} \subseteq \DM(Y)$.    Then Proposition \ref{27} implies that:
\begin{align*}
  i\left(\bigwedge{}^{\!\DM(Y)}_{\alpha\in\mathcal A}A_\alpha\right)= &  \left(\bigcap_{\alpha \in \mathcal{A}} A_{\alpha}\right)^{+-} \\
  = &\bigcap_{\alpha \in \mathcal{A}} A_{\alpha}^{+-} \\
  = &\bigwedge{}^{\!\DM(L)}_{\alpha\in\mathcal A}\ i(A_\alpha)\,,
\end{align*}
if $Y$ satisfies Property (A); and
\begin{align*}
  i\left(\bigvee{}^{\!\DM(Y)}_{\alpha\in\mathcal A} A_\alpha\right) =&\left (\bigcup_{\alpha \in \mathcal{A}} A_{\alpha}\right)^{+_{Y}-_{Y}+-}  \\
  = & \left(\bigcup_{\alpha \in \mathcal{A}} A_{\alpha}^{+-}\right)^{+-} \\
  = & \bigvee{}^{\!\DM(L)}_{\alpha\in\mathcal A}\  i(A_\alpha)\,,
\end{align*}
if $Y$ satisfies Property (B).
\end{proof}

\begin{cor}\label{31}
Let $L$ be a lattice and $Y$ be a sublattice satisfying Properties (A) and (B). Then $i[\DM(Y)]$ is a regular sublattice of $\DM(L)$.
\end{cor}

\begin{prop}\label{32}
	Let $L$ be a lattice and $Y$ be a convex sublattice.
	\begin{enumerate}[(i)]
		\item If $Y$ has a maximal element, it satisfies Property (B);
		\item If $Y$ has a minimal element, it satisfies Property (A).
	\end{enumerate}
\begin{proof}
	We will prove (i), as (ii) can be proved dually. Let $a_{\max}$ be the maximal element of $Y$. Let $A \subseteq Y$ and $x \in A^{+}$. Then by convexity, $a_{\max} \wedge x \in Y$ and $a_{max} \wedge x \in A^{+_{Y}}$.
\end{proof}	
\end{prop}

\begin{prop}\label{35}
	Let $L$ be a complete lattice and $Y$ a sublattice of $L$. Then $Y$ is $O$-closed if and only if $Y$ is supremum and infimum closed.
\end{prop}

\begin{proof}
	Assume that $Y$ is O-closed. Let $A \subseteq L$ with $\lbigvee{L} A = x$. Let $\mathscr{F} = \{B \subseteq A : \vert B \vert < \omega \}$.  $\mathscr{F}$ is directed with respect to inclusion. For every $B \in \mathscr{F}$ let $x_{B} = \lbigvee{L}B=\lbigvee{Y}B$ to get an increasing net $\{x_{B} : B\in \mathscr{F} \}$ in $Y$ with $\lbigvee{L}_{B\in\mathscr F}x_{B} = x$, i.e. a net in $Y$ that O-converges to $x$.  Since $Y$ is O-closed,  $x \in Y$. The dual can be proved similarly.
		
	Conversely, assume that for any subset $A$ of $Y$, its supremum and infimum in $L$ belong to $Y$, and let $x \in Y^{\sss O}_{1}$. Then there exists a net $(x_{\gamma})_{\gamma \in \Gamma}$ such that
	\begin{equation*}
		\bigvee{}^{\!L}_{\gamma'\in\Gamma}\bigwedge{}^{\!L}_{\gamma\ge \gamma'} x_{\gamma} = x = 	\bigwedge{}^{\!L}_{\gamma'\in\Gamma}\bigvee{}^{\!L}_{\gamma\ge \gamma'}x_\gamma\,.
	\end{equation*}
	By hypothesis it follows that $x \in Y$.
\end{proof}

\medskip

\begin{prop}\label{36}
	Let $L$ be a complete lattice and $Y$ a complete regular sublattice of $L$. Then $Y$ is O-closed.
\end{prop}
\begin{proof}
	The result follows by Proposition \ref{35} because $Y$ is supremum and infimum closed.
\end{proof}	

\begin{thm}\label{n}
Let $L$ be an infinitely distributive lattice and $Y \subseteq L$ a sublattice satisfying Properties (A) and (B). Then $Y^{\sss{\mathfrak{u}{O}}}_1=Y^{\sss{{O}}}_1$, and both are simultaneously O-closed and $\mathfrak{u}$O-closed.
\end{thm}
\begin{proof}
Identify $\DM(Y)$ with $i[\DM(Y)]\subseteq \DM(L)$.   Our hypothesis and Corollary \ref{31} imply that $\DM(Y)$ is a complete regular sublattice of $\DM(L)$.  In particular, by Proposition \ref{36}, $\DM(Y)$ is O-closed (and therefore $\mathfrak{u}$O-closed because $\DM(L)$ has  maximal and minimal elements).  We show that $\DM(Y) \cap L =
	Y^{\sss O}_{1}$.   For every $x \in Y^{\sss O}_{1}$ there exists a net
	$(x_\gamma)_{\gamma\in\Gamma}$ in $Y$ such that $x_{\gamma} \converges{O}
	x$ in $L$. From \cite[Theorem 3]{KAECWH} we have that $x_{\gamma} \converges{O} x$ in $\DM(L)$.  Since $\DM(Y)$ is O-closed in $\DM(L)$, it follows that  $x \in \DM(Y)\cap L$. To see the converse, take $x \in \DM(Y) \cap L$. Let $A = \{y \in Y : y \leq x \}$.  By Theorem \ref{30}, $\DM(Y)$ is regular in $\DM(L)$ and thus$\lbigvee{\DM(L)}A=\lbigvee{\DM{Y}}A= x$.  Since $L$ is regular in $\DM(L)$, it follows that $x\in Y^{\sss O}_{1}$.  This shows that $\DM(Y) \cap L =
	Y^{\sss O}_{1}$, and therefore $Y^{\sss O}_{1}$ is simultaneously O-closed and $\mathfrak{u}$O-closed.
\end{proof}

   \def\bibindent{1em}

\end{document}